\documentclass[11pt]{article}

\usepackage{amsmath,amssymb,amsthm,amsfonts,bbm}
\usepackage{latexsym, comment}
\usepackage[square,numbers]{natbib}
\usepackage{tikz, xcolor}
\usepackage{authblk}

\numberwithin{equation}{section}

\allowdisplaybreaks

\newtheorem{theorem}{Theorem}[section]

\newtheorem{prop}[theorem]{Proposition}
\newtheorem{cor}[theorem]{Corollary}
\newtheorem{lem}[theorem]{Lemma}

\theoremstyle{definition}

\newtheorem{remark}[theorem]{Remark}
\newtheorem{rem}[theorem]{Remark}

\theoremstyle{remark}

\newenvironment{romenumerate}[1][0pt]{
\addtolength{\leftmargini}{#1}\begin{enumerate}
 }{\end{enumerate}}

\newcounter{oldenumi}
{\setcounter{oldenumi}{\value{enumi}}
\begin{romenumerate} \setcounter{enumi}{\value{oldenumi}}}
{\end{romenumerate}}

\newcounter{thmenumerate}

\newcounter{romxenumerate}   

\newcounter{xenumerate}   

\newcounter{pfcase}

\newcommand{\refS}[1]{Section~\ref{#1}}

\newcommand\set[1]{\ensuremath{\{#1\}}}

\newcommand\bigpar[1]{\bigl(#1\bigr)}

\def\rompar(#1){\textup(#1\textup)}    

\def\xexp(#1){e^{#1}}

\newcommand\ntoo{\ensuremath{{n\to\infty}}}

\newcommand\punkt[1]{\if.#1\else.\spacefactor1000\fi{#1}}

\newcommand\ie{i.e\punkt}
\newcommand\eg{e.g\punkt}

\newcommand\whp{w.h.p\punkt}

\newcommand{\tend}{\longrightarrow}

\newcommand\pto{\overset{\mathrm{p}}{\tend}}

\newcommand\op{o_{\mathrm p}}
\newcommand\Op{O_{\mathrm p}}

\newcounter{CC}
\newcounter{cc}

\newcommand\E{\operatorname{\mathbb E{}}}
\renewcommand\P{\operatorname{\mathbb P{}}}
\newcommand\Var{\operatorname{Var}}

\newcommand\Po{\operatorname{Po}}

\newcommand\Bin{\operatorname{Bin}}

\newcommand\Be{\operatorname{Be}}

\newcommand\cA{\mathcal A}

\newcommand\cG{\mathcal G}

\newcommand\cL{{\mathcal L}}

\newcommand\cN{\mathcal N}

\newcommand\cQ{\mathcal Q}
\newcommand\cR{{\mathcal R}}

\newcommand\cZ{{\mathcal Z}}

\def\[#1]{[\![#1]\!]}

\renewcommand{\=}{:=}

\newcommand\gnp{\ensuremath{G_{n,p}}}

\newcommand\cao{\cA(0)}
\newcommand\cDA{\Delta\cA}
\newcommand\Ax{A^*}

\newcommand\etto{\bigpar{1+o(1)}}
\newcommand\ettop{\bigpar{1+\op(1)}}

\newcommand\simin{\in}

\newcommand{\1}{{\mathbbm 1}}

\newcommand\REM[1]{{\raggedright\texttt{[#1]}\par\marginal{XXX}}}

\setbox2=\hbox{r}

\title{Majority bootstrap percolation on the random graph $\gnp$}

\author[*]{
Sigurdur  \"Orn Stef\'ansson}
\author[*]{Thomas Vallier }
\affil[*]{University of Iceland}

\date{}
\begin{document}
\maketitle

\begin{abstract}
Majority bootstrap percolation on the random graph $G_{n,p}$
is a process of spread of
``activation''
on a given realisation of the graph with a given number of
initially active nodes. At each step those vertices which have more
active neighbours than inactive neighbours become active as well.

We study the size $A^*$ of the final active set. The parameters of the model
are, besides $n$ (tending to $\infty$), the size $A(0)=A_0(n)$ of
the initially active set and the
probability $p=p(n)$ of the edges in the graph. 
We prove that the process cannot percolate for $A(0) = o(n)$.
We study the process for $A(0) = \theta n$ and every range of $p$ and show that the model
exhibits different behaviours for different ranges of $p$.
For very small $p \ll \frac{1}{n}$, the activation does not spread significantly. For large $p \gg \frac{1}{n}$ then we see a phase transition at $A(0) \simeq \frac{1}{2}n$. In the case $p= \frac{c}{n}$, the activation propagates to a significantly larger part of the graph but (the process does not percolate) a positive part of the graph remains inactive.
\end{abstract}

\section{Introduction}

Majority bootstrap percolation on a graph $G$ is defined as the spread of
\emph{activation} or \emph{infection} according to
the following rule:
We start with a set $\cao\subseteq V(G)$ of
\emph{active} vertices.
Each inactive vertex that has more active neighbours than inactive
becomes active. This is repeated until no more vertices become active.
Active vertices never become inactive, so the set of active vertices
grows monotonically.

We are mainly interested in the final size $|\cA^*| = \Ax$ of the active set on the random graph $\gnp$, and in
particular whether eventually all vertices will be active or not.
If they are, we say that the initial set $\cao$
\emph{percolates} (completely). We will study a sequence of graphs of order
$n\to\infty$; we then also say that (a sequence of)
$\cao$ \emph{almost percolates} if the number of
vertices that remain inactive is $o(n)$, \ie, if $\Ax=n-o(n)$.
In both cases, we talk about supercritical phase. If the activation does not spread to almost all the graph then we talk about subcritical phase.

Recall that $\gnp$ is the random graph on the set of
vertices $V_n = \{ 1,\dots,n\}$ where all possible edges between pairs of
different vertices are present independently and with the same
probability $p$.

The problem of majority bootstrap
percolation where a vertex becomes activated if at least half of its
neighbours are active ($r(v) = \deg(v)/2$) has been studied on the hypercube $\cQ_n = [2]^n$ by
Balogh, Bollob{\'a}s and Morris \cite{BBM}. They consider the case when vertices are set as active at time $0$ independently with a certain probability $q_n$. The main result of \cite{BBM} states that the critical probability is $q_c(\cQ_n) = \frac{1}{2}$. More precisely, they also determine the second order term of the critical probability. If
\begin{equation}
q(n) = \frac{1}{2} - \frac{1}{2} \sqrt{\frac{\log n}{n}} + \frac{\lambda \log \log n}{\sqrt{n \log n}},
\end{equation}
then
\begin{equation}
\P \left\{ \cA^* = \cQ_n\right\} \to 
\begin{cases}
0 \quad \text{if } \lambda \leq -2
\\
1 \quad \text{if } \lambda > \frac{1}{2}.
\end{cases}
\end{equation}
Those results can be compared to our Corollary \ref{cor:p>>1overn} where we prove that for highly connected graphs, the transition happens for $q=1/2$.

The model of global cascade on random networks which generalises majority bootstrap percolation as one requests a proportion $0<\alpha<1$ of the neighbours to be active has been introduced by Watts in \cite{W}. 
The case $\alpha =1/2$ is the majority bootstrap percolation.
The author of \cite{W} derives conclusions using assumptions on the internal structure of the network
 from numerical simulations on randomly generated networks of 1000 nodes.
Our results agree qualitatively as low connectivity limits the propagation of the activation by the lack of connection. We show in Theorem \ref{theo:o1overn} that for $p =o(1/n)$, no propagation is possible w.h.p.
Moreover Watts notices that the propagation is limited by the stability of the nodes in dense graphs. We show in Theorem \ref{theo:p>>1overn} that for $p\gg 1/n$, the critical size for percolation is $A_c=\frac{1}{2} n + o(n)$.

We provide an analytical treatment of the problem of majority bootstrap percolation on the graph $\gnp$. Our results extend to the case of global cascade which we rename as proportional bootstrap percolation with parameter of proportionality $\alpha$.

The authors of \cite{JLTV} studied (the classical) bootstrap percolation on the
Erd\"os--R\'enyi random graph $\gnp$
with an initial set $\cao$ consisting of a given number $A(0)$ of vertices
chosen at random.  
In the classic bootstrap percolation, a vertex becomes active if it has at least $r \geq 2$ incoming activations.

They prove that there is a threshold phenomenon: 

For $p \gg \frac{1}{n}$ then typically, either
the final size
$\Ax$ is small, $\Ax = o_p(n)$ (at most twice the initial size $A(0)$),
or it is large, $\Ax = n- o_p(n)$ (sometimes exactly $n$, but if $p$ is so
small that there are vertices of degree less than $r$, these can never
become active except initially so eventually at most $n-o(n)$
will become infected). 

That result can be related with our Theorem \ref{theo:p>>1overn} to compare classical and majority bootstrap percolation.

In the case of $p= \frac{c}{n}$, the authors of \cite{JLTV} prove that w.h.p. only the activation starting from a significant part of the graph $A(0) = \theta n$, $\theta >0$ spreads to a larger part of the graph but not all the graph, in which case $\Ax = \theta^* n$, $\theta < \theta^*<1$ where $\theta^*$ is exactly and uniquely determined as the smallest root larger than $\theta$ of a given equation.

We prove here, in the case of majority bootstrap percolation, for $p = \frac{c}{n}$ that similarly, the activation spreads to a larger part of the graph so that $A^* = \theta^*n$ with $\theta < \theta^*<x_0 <1$ where $x_0\geq \theta$ is the smallest root of the equation \eqref{eq:root} satisfying \eqref{eq:defx0theo}. See Theorem \ref{theo:covern} in Section \ref{Sresults} 

One may notice that in the case of bootstrap percolation with threshold $r>1$, no vertex of degree $r-1$ can be activated. That immediately eliminates the vertices of degree 1.
Therefore, vertices of degree 1 never become active unless they are set as active at the origin.
Conversely, in the case of majority bootstrap percolation, any vertex of degree 1 that has a link to an active vertex becomes active.

\begin{remark}
An alternative to starting with an initial active set of fixed size $A(0)$
is to let each vertex be initially activated with probability $q=q(n)>0$,
with different vertices activated independently. Note that
this is the same as taking
the initial size $A(0)$ random with $A(0) \simin\Bin(n,q)$.

Therefore, our results can be translated from one case to the other.
\end{remark}
\subsection{Notation}

All unspecified limits are as \ntoo.
We use $\Op$ and $\op$ in the standard sense (see \eg{}
\cite{JLR} and \cite{SJN6}), and we use \whp{} (with high
probability) for events with probability tending to 1 as \ntoo.
Note that, for example,  `$=o(1)$ \whp' is equivalent to `$=o_p(1)$'
and to `$\pto0$'
(see \cite{SJN6}).
We denote $\cN_v $ the neighbourhood of a vertex $v$ and $|\cN_v| = \deg(v)$ its degree.
The notation $f \gg g $ means that $g = o(f)$, for example $p \gg \frac{1}{n}$ is equivalent to $\lim n p = + \infty$ or that there exists a function $\omega(n)$ with $\lim_{n \to \infty} \omega(n) = + \infty$ with $p = \frac{\omega(n)}{n}$ with the implicit condition that $\omega(n) \leq n$ for definiteness of $p \leq 1$.


The method is described in \refS{Ssetup}.
The main results are stated in \refS{Sresults}.
Preliminary results are derived in \refS{Sprob} and \refS{Sbound}.
\refS{Sproof1}--\refS{Sproof3} are dedicated to the proofs.

\section{Reformulation of the process}\label{Ssetup}

We use an algorithm to reveal the vertices activated that resembles the one from \cite{JLTV}.

In order to analyse the bootstrap percolation process on
$\gnp$, we change the time scale; we consider
at each time step
the activations from one vertex only.
Choose $u_1\in\cao$ and give each of its neighbours a \emph{mark};
we then say
that $u_1$ is \emph{used}, and let
$\cZ(1)\=\set{u_1}$ be the set of used vertices
at time 1. At some time $t$, let
$\cDA(t)$ be the set of inactive vertices with the number of
marks larger than half their degree; these now become active and we let
$\cA(t)=\cA(t-1)\cup\cDA(t)$ be the set of
active vertices at time $t$.  Denote by $\cZ(t-1)$ the set of vertices which have been used at time $t-1$.
We continue recursively: At time $t \leq A(t) = |\cA (t)|$, choose a vertex
$u_{t}\in\cA(t)\setminus\cZ(t-1)$.
We give each neighbour of $u_{t}$ a new mark. We keep the unused, active vertices
in a queue and choose $u_{t}$ as the first vertex in the queue.
The vertices in $\cDA(t)$ are added at the end of
the queue in order of their labels. 
Using this setting, the vertices are explored one at a time in the order of their activation or appearance in the set of active vertices.

We finally set
$\cZ(t)=\cZ(t-1)\cup\set{u_{t}}=\set{u_s:s\le t}$,
the set of used vertices. (We start with
$\cZ(0)=\emptyset$.)

The process stops when
$\cA(t)\setminus\cZ(t)=\emptyset$, \ie, when
all active vertices are used. We denote this stopping time by $T$,
\begin{equation}
  \label{t1}
T\=\min\set{t\ge0:\cA(t)\setminus\cZ(t)=\emptyset}.
\end{equation}
Clearly, $T\le n$. In particular, $T$ is finite.
The final active set is $\cA(T)$. It is clear that
this is the same set as the one produced by the bootstrap percolation
process defined in the
introduction, only the time development differs.

Let $A(t)\=|\cA(t)|$, the number of active vertices at time $t$.
Since $|\cZ(t)|=t$ and $\cZ(t) \subseteq \cA(t)$
for $t=0,\dots,T$, we also have
\begin{equation}
  \label{t2}
T=\min\set{t\ge0:A(t)=t}
=
\min\set{t\ge0:A(t)\le t}.
\end{equation}
Moreover,
since the final active set is  $\cA(T)=\cZ(T)$,
its size $\Ax$ is
\begin{equation}\label{at}
\Ax\=A(T)=|\cA(T)|=|\cZ(T)|=T.
\end{equation}
Hence, the set $\cao$ percolates if
and only if $T=n$, and $\cao$ almost percolates if and only if $T=n-o(n)$.
\begin{remark}
In order to find the final set of active vertices, it is not important in which order we explore the vertices.
However, the fact that a vertex $v$ has been activated at a certain time $y$ has 
incidence on its connectivity to the set of inactive vertices $\cR(t) = V \setminus \cA(t)$. 
The condition 
\begin{equation}\label{eq:condideg}
|\cN(v) \cap \cZ(y)| \geq \max\left( |\cN(v) \cap V \setminus \cZ(y)|;1 \right)
\end{equation}
 has to be fulfilled for $v$ to be active at time $y$.
\end{remark}
Let $p_s$ denote the probability that a vertex $i \in V \setminus \cZ(s)$ receives a mark at time $s >A(0)$, 
\begin{equation*}
p_s  = \P \left\{ \left\{ | \cZ(s)| =s \right\} \cap (u_s,i) \right\}
\end{equation*}
where $u_s$ is a way to denote the vertex in $\cA(s)$ which is explored at time $s$ and $|\cZ(s)| = s$ means that the algorithm has not stopped at time $s$. 

We immediately derive the following simple but useful bounds on the probability that a vertex receives an incoming activation from the vertex $u_s$ at time $s$
\begin{equation}\label{eq:key}
p_s \leq p ,
\end{equation}
for any $u_s \notin \cA(0)$.

For $u_s \in \cA(0)$, that is $s \leq A(0)$, the condition $\{| \cZ (s) | = s \}$ is fulfilled and thus we have
\begin{equation}
p_s = p \quad \text{for } s \leq A(0).
\end{equation}
 
Let $\1_i(s)$ be the indicator that $i$ receives a mark at time $s$  i.e. there is an edge between $u_s$ and $i$. We have 
\begin{equation}\label{eq:indics}
\1_i(s) \in \Be(p_s).
\end{equation}
For $s \leq A(0)$ or equivalently $u_s \in \cA(0)$, this is also the indicator that there is an edge between the vertices $u_s$ and $i$. Thus 
\begin{equation}\label{eq:indic0}
\1_i(s) \in \Be(p) \quad \text{for } s \leq A(0).
\end{equation}
and the variables are independent for different $s \leq A(0)$.

Let $M_i(t)$ denote the number of marks $i$ has
at time $t$, then
\begin{equation}
  \label{mi}
M_i(t)=\sum_{s=1}^t \1_i(s),
\end{equation}
at least until the vertex $i$ is activated (and what happens later does not matter).
Note that if $i\notin\cao$, then, for every $t\le T$,
$i \in\cA(t)$ if and only if $M_i(t)\ge \frac{\deg(i)}{2}$.

The sequence of random variables $M_i(t)$ is the number of marks that a vertex receives.
Our focus is to find the number of vertices for which the number of marks is larger than $1/2$ of their degree.
Therefore, being connected to an active vertex that has been explored and being connected to an active vertex that has not yet been explored is very different.
Until a vertex has been explored, its activeness has not been revealed to its neighbours.
That algorithm does not change the final size of the set of active vertices. Indeed, all the vertices will be explored eventually and moreover the fact that $v$ becomes active is a monotonic increasing function of the number of active vertices that have a link with $v$.
\begin{figure}[h]\label{fig:activation}
\begin{center}
\begin{tikzpicture}[scale=1.3, every to/.style={bend right}]


\foreach \nodes in {0.5,1,...,3}
\node at (\nodes,0) [circle, draw=black]{};

\foreach \nodes in {0.5,1,...,3}
\draw (\nodes,0) circle(2pt);

\foreach \nodes in {0.5,1,...,5}
\shade[ball color=black](\nodes,0) circle(2pt);

\shade[ball color=gray, opacity=0.6](6,1) circle(2pt);
\shade[ball color=gray, opacity=0.6](7,0.5) circle(2pt);
\shade[ball color=gray, opacity=0.6](7.5,0) circle(2pt);
\shade[ball color=gray, opacity=0.6](6.8,-0.4) circle(2pt);
\shade[ball color=gray, opacity=0.6](6.2,-0.2) circle(2pt);
\shade[ball color=gray, opacity=0.6](6.5,0.4) circle(2pt);
\shade[ball color=gray, opacity=0.6](6.7,0.9) circle(2pt);
\shade[ball color=gray, opacity=0.6](7.4,0.6) circle(2pt);
\draw (3,-0.1) node[below] {\footnotesize{$t$}};
\draw (5.1,-0.1) node[below] {\footnotesize{$A(t)$}};
\draw (4,-0.1) node[below] {\footnotesize{$i$}};
\draw (6.5,-0.1) node[below] {\footnotesize{$l$}};

\draw[dashed] (6.2,-0.2) to (7,0.5) ;
\draw[-](6.2,-0.2)to(1.5,0);
\draw[dashed](6.2,-0.2)to(4,0);
\shade[ball color=gray, opacity=0.6](1,-1) circle(2pt);
\draw (1.1,-1) node[right] {\footnotesize{: inactive vertices}};
\shade[ball color=black](1,-1.5) circle(2pt);
\draw (1.1,-1.5) node[right] {\footnotesize{: activated vertices}};
\shade[ball color=black](1,-2) circle(2pt);
\node at (1,-2) [circle, draw=black]{};
\draw (1.1,-2) node[right] {\footnotesize{: activated and explored vertices}};

\draw[-](7,-1) to (6,-1);
\draw (7.05,-1) node[right] {\footnotesize{: excitatory links}};
\draw[dashed] (7,-2) to (6,-2);
\draw (7.05,-2) node[right] {\footnotesize{: inhibitory links}};

\end{tikzpicture}
\end{center}
\caption{At time $t$, the vertex $l$ is not yet activated.}\label{graph:gnpact}
\end{figure}
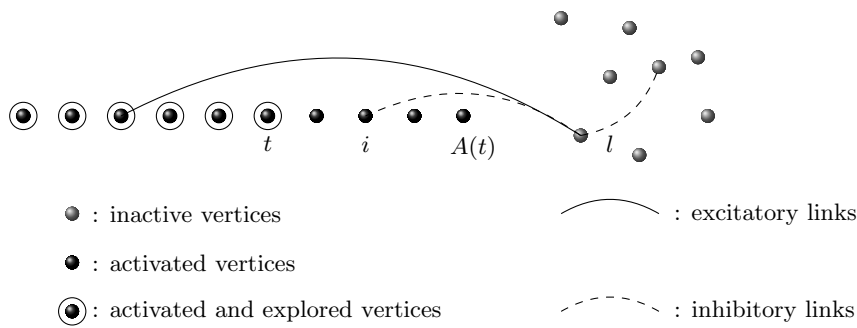

%
%
Define also, for $i\in V_n\setminus\cao$,
\begin{equation}
  \label{yi}
Y_i\= \min \{t: M_i(t) \ge  \frac{1}{2} \deg(i) \cap M_i(t) >0\}.
\end{equation}
If $Y_i\le T$, then $Y_i$ is the time vertex $i$ becomes active,
but if $Y_i>T$, then $i$ never becomes active. Thus, for $t\le T$,
\begin{equation}
  \label{at2}
\cA(t)=\cao\cup\set{i\notin\cao:Y_i\le t}.
\end{equation}
Denote $I_i(t) =  {\mathbbm 1}_{\{Y_i\le t\}} $, the indicator function that the vertex $i$ is active at time $t$
and let
\begin{equation*}\label{eq:pi(t)}
\pi(t) = \P \left\{ I_i(t) = 1 \right\}.
\end{equation*}
The probability $\pi(t)$ is independent of $i$.
We let, for $t=0,1,2,\dots$,
\begin{equation}\label{st}
  S(t)\=|\set{i\notin\cao:Y_i\le t}|
= \sum_{i\notin\cao} {\mathbbm 1}_{\{Y_i\le t\} }  = \sum_{i\notin\cao} I_i(t),
\end{equation}
so, by \eqref{at2} and our notation,
\begin{equation}\label{as}
  A(t)=A(0)+S(t).
\end{equation}
By the relations \eqref{t2}, \eqref{at} and \eqref{as} it suffices to study the process $S(t)$.
$S(t)$ is a sum of identically distributed processes $ I_i(t) \in \Be \left(\pi(t)\right)$.
The main problem is that we do not have independence of the random variables $I_i$, $i=1,...,n-A(0)$.
Take any two vertices $i$ and $j$. The probability that the vertex $i$ is activated depends on its degree and therefore on having or not a connection to $j$. The activation of the vertex $i$ therefore gives an indication on the existence or not of an edge $(i,j)$. Thus this gives indications whether the vertex $j$ is active.


That implies that the random variable $S(t) = \sum_{i\notin\cao} I_i(t)$ is not a sum of independent Bernoulli random variable and hence is not a binomial.

Though the random variables $I_i (t)= {\mathbbm 1}_{\{Y_i \leq t\}}$ are not independent, they are very close to being independent since the dependency between two random variables $I_i(t)$ and $I_j(t)$ is only through the possible connection $\{i,j\}$.

Let $R(t) = n-A(t)$ denote the number of inactive vertices. 
It is equivalent to study $R(t)$ which is also a sum of identically distributed Bernoulli random variables
\begin{equation}\label{eq:defrt}
R(t) = \sum_{i=1}^{n-A(0)} 1- I_i(t) = \sum_{i=1}^{n-A(0)} K_i(t),
\end{equation}
where $K_i(t) \in \Be \left(1-\pi(t) \right)$. We shall denote $\delta(t) = 1 - \pi(t)$ so that $K_i(t) \in \Be \left( \delta(t) \right)$.

The proofs of the supercritical case rely on proving that $R(t) =o_p(n)$.
\section{Results}\label{Sresults}
We give the results depending on the value of $p$.

When $p = o \left( \frac{1}{n} \right)$ then there are too few connections for the activation to spread 
\begin{theorem}\label{theo:o1overn}
If $p = o \left(\frac{1}{n}\right)$, then for any $\varepsilon >0$, we have
\begin{equation}\label{eq:o1overn}
\lim_{n \to \infty} \P \left\{ A^* > (1+ \varepsilon) A(0)\right\} = 0,
\end{equation}
that is
\begin{equation*}
A^* = A(0) \ettop.
\end{equation*}
\end{theorem}
In the case when $p = \frac{c}{n}$ and if $\cA(0)$ contains a positive part of the graph, then the activation spreads to a larger part of the graph but does not completely percolate.
\begin{theorem}\label{theo:covern}
If $p = \frac{c}{n}$ for some $0<c<\infty$, we have
\begin{romenumerate}
\item \label{theo:coverni}
If $A(0) = o(n)$, let $g(c)= (1+c) c e^{-c}$ then 
\begin{equation}\label{eq:coverni}
A^* = o_p(n),
\end{equation}
more precisely, we have for $A(0) \to \infty$ as $n \to \infty$
\begin{equation}\label{eq:coverniprecise}
A^* \leq \frac{1}{1-g(c)} A(0) \ettop.
\end{equation}
%
\item \label{theo:covernii}
If $A(0) = \theta n$, for some $0< \theta<1$, then we have
\begin{equation}\label{eq:covernii}
A^* = \theta^* n + o_p(n),
\end{equation}
with $\theta < \theta^*\leq x_0 <1$ where 
\begin{equation}\label{eq:defx0theo}
x_0 = \inf \{x \geq \theta,  f_{c,\theta}(x)< 0\},
\end{equation}
with
\begin{equation}\label{eq:root}
f_{c,\theta}(x) = \theta  - x +  \left( 1 - \theta \right) e^p e^{-c} \sum_{k=1}^{\lfloor xn \rfloor} \frac{(cx)^k}{k!} \sum_{j=0}^k \frac{\left((1-x)c\right)^j}{j!}.
\end{equation}
\end{romenumerate}
\end{theorem}
\begin{remark}
Even though $x_0$ depends on $n$, it has a limit strictly less than $1$ as $n \to \infty$.
\end{remark}
\begin{remark}
Notice in the case of Theorem \ref{theo:covern} \ref{theo:coverni} that 
\begin{equation*}
\lim_{c \to 0} g(c) = 0  \text{ and } \lim_{c \to \infty} g(c) = 0.
\end{equation*}
These limits are consistent with the results of Theorems \ref{theo:o1overn} and \ref{theo:p>>1overn} \ref{theo:p>>1overni}. 
One should remark also that even though $A(0) =o(n)$, the vertices of degree $1$ and $2$ may contribute to enlarge the set of activated vertices. The vertices of higher degree tend to be more stable as is seen in the following theorem.
\end{remark}

If one increases the connectivity such that $p \gg \frac{1}{n}$ then the high number of connections tends to stabilise the process such that the threshold for majority bootstrap percolation is at $A(0) = \frac{1}{2}n$.
\begin{theorem}\label{theo:p>>1overn}
If $\frac{1}{n} \ll p \leq 1$ then
\begin{romenumerate}
\item \label{theo:p>>1overni}
If $A(0) = o(n)$ is monotonically increasing in $n$ then
\begin{equation}
A^* = o_p(n)
\end{equation}
More precisely
\begin{equation}\label{eq:p>>1overni}
\begin{cases}
 A^* = A(0) \ettop  & if A(0) \gg n \exp\left( -\frac{1}{3} np\right)
\\
 A^* = O_p \left( n \exp\left( -\frac{1}{3} np\right) \right)  & \text{if } A(0) \leq K n \exp\left(-\frac{1}{3} np\right) \text{ for some } K>0. 
\end{cases}
\end{equation}
\item  \label{theo:p>>1overnii}
If $A(0) = \theta n$, $0 < \theta < \frac{1}{2}$ then  
\begin{equation}\label{eq:p>>1overnii}
A^* = A(0) \ettop.
\end{equation}
\item  \label{theo:p>>1overniii}
If 
\begin{equation}
\lim_{n \to \infty} \frac{A(0) -\frac{1}{2}n}{\sqrt{\frac{n}{p}}} = + \infty
\end{equation}
then 
\begin{equation}\label{eq:p>>1overnii2}
A^* = n - o_p(n).
\end{equation}
\end{romenumerate}
\end{theorem}
Notice that for example, the statement of equation \eqref{eq:p>>1overnii} is equivalent to
\begin{equation}\label{eq:p>>1overniibis}
\lim_{n \to \infty} \P \left\{ A^* \geq (1+\varepsilon) A(0) \right\} = 0,
\end{equation}
We give here the counterpart of Theorem \ref{theo:p>>1overn} using the setting of \cite{BBM}, that is, when the vertices are initially activated independently with some probability $q$.
\begin{cor}\label{cor:p>>1overn}
Let $\frac{1}{n} \ll p \leq 1$. Suppose that the vertices are initially activated independently with probability $q \in (0,1)$.
\begin{romenumerate}
\item \label{cor:p>>1overni}
 If $q < 1/2$ then
\begin{equation}
\Ax = A(0) \ettop.
\end{equation}
\item \label{cor:p>>1overnii}
 If $q >1/2$ then
\begin{equation}
\Ax = n - o_p(n).
\end{equation}
\end{romenumerate}
\end{cor}
\begin{proof}[Proof of Corollary ]
Let $\lambda>0$ and let $q <\frac{1}{2}$ then the number of vertices initially active is 
\begin{equation}
A(0) \in \Bin\left( n, q \right).
\end{equation}
We know that $\Var \left( A(0)\right)  \leq \E \left( A(0)\right) = nq$ so using Chebyshev's inequality, we find that for any  $ 0< \lambda < \frac{1}{2}-q$ 
\begin{equation}\label{eq:translate}
\lim_{n \to \infty} \P \left\{ A(0) \geq \left( q + \lambda \right) n \right\} = 0.
\end{equation}
By use of Theorem \ref{theo:p>>1overn} \ref{theo:p>>1overnii} and equation \eqref{eq:translate} we find that
\begin{align*}
\lim_{n \to \infty} \P \left\{ \Ax > (1+\epsilon) A(0) \right\}
&  \leq \lim_{n \to \infty} \P \left\{ \Ax > (1+\epsilon) A(0)\  \Big| \  A(0) \leq \left( q + \lambda \right) n \right\} 
 \\ 
 & \qquad  + \lim_{n \to \infty} \P \left\{ A(0) \geq \left( q + \lambda \right) n \right\} 
 \\
 & = 0.
\end{align*}
That proves corollary \ref{cor:p>>1overn} \ref{cor:p>>1overni}. The item \ref{cor:p>>1overnii} can be proved similarly using Theorem \ref{theo:p>>1overn} \ref{theo:p>>1overnii} and concentration results on the binomial random variable. 
\end{proof}

\section{Probability of activation of a vertex}\label{Sprob}
We start by determining the probability of activation of a vertex $i \in V \setminus \cA(0)$ as it will be needed all along the article,
\begin{equation*}
\pi(t) = \P \left\{ Y_i \leq t \right\}.
\end{equation*}
We use the notation  
\begin{equation}\label{eq:bindeg}
\Bin_i([1,n],p) \in \Bin(n-1,p)
\end{equation}
 to denote the degree of the vertex $i$, that is a sum of Bernoulli $\Be(p)$ independent random variables corresponding to the existence of an edge to another vertex. We denote
\begin{equation}\label{eq:binn-t-1}
\Bin_i([t+1,n],p) \in \Bin(n-t-1,p),
\end{equation} 
the number of links that the vertex $i$ has to the set $\{t,...,n\} = V \setminus \cZ(t)$.
The random variables $\Bin_i([1,t],p) $ and $\Bin_i([t+1,n],p)$ are independent as they concern summations of independent Bernoulli random variables on disjoint sets.
The number of links of the vertex $i$ to the set of vertices $\{1,...,t\} = \cZ(t)$ constructed in the algorithm is denoted $M_i(t)$.
Remark that the equality $M_i(t) \in \Bin(t,p)$ is in general not true 
because the vertices of $\cZ(t) \setminus \cA(0)$ need to verify the condition \eqref{eq:condideg}. 
In the special case when $t\leq A(0)$ then the condition \eqref{eq:condideg} does not need to be fulfilled.Therefore, we have $M_i(t) \in \Bin(t,p)$ for $t\leq A(0)$.
In the following, we abuse notations and write for example $\Bin(t,p)$ for a random variable with binomial distribution $\Bin(t,p)$.


Since a vertex only accumulates marks, we have
\begin{align}\label{eq:pitdef}
\pi(t) & = \P \left\{ M_i(t) \geq \max \left( \frac{1}{2} \deg(i) ; 1\right) \right\}
\nonumber
\\
& = \P \left\{ M_i(t) \geq \max \left( \frac{1}{2} \Bin_i([1,n],p) ; 1\right) \right\}
\nonumber
\\
& = \P \left\{ \sum_{s=1}^t \1_i (s) \geq \max \left( \frac{1}{2} \Bin_i([1,n],p) ; 1\right) \right\}.
\end{align}
The probability of activation can also be rewritten
\begin{align*}\label{eq:alternatepitdef}
\pi(t) 
& = \P \left\{ M_i(t) \geq \max \left( \Bin_i([t+1,n],p) ; 1\right) \right\}
\\
& = \P \left\{  \sum_{s=1}^t \1_i (s) \geq \max \left( \Bin_i([t+1,n],p) ; 1\right) \right\}.
\end{align*}
\begin{lem}\label{lem:stocdomin}
The random variable $M_i(t)$ is stochastically  dominated by $\Bin(t,p) $.
\end{lem}
\begin{proof}[Proof of Lemma \ref{lem:stocdomin}]
\begin{equation}
\P \left\{ M_i(t) \geq k \right\} = \P \left\{ \sum_{s=1}^t \1_i(s) \geq k \right\}.
\end{equation}
Let $\cL_k$, with $|\cL_k| =k$, be some 
subset of $\{1,..., t\}$. Then 
\begin{align*}
\P \left\{ \sum_{s=1}^t \1_i(s) \geq k \right\}
& = \P \left\{ \bigcup_{\cL_k \subseteq \{1,...,t\}} \left( \sum_{j \in \cL_k} \1_i(j) =k \cap \sum_{j \notin \cL_k} \1_i(j) \geq 0  \right) \right\},
\end{align*}
where the event $ \left\{ \sum_{j \notin \cL_k} \1_i(j) \geq 0 \right\}$ is always fulfilled as the random variable $\1_i(j)$ can take only the values $0$ and $1$.
Moreover,
\begin{equation}
\P \left\{\sum_{j \in \cL_k} \1_i(j) =k \right\} = \P \left( \bigcap_{j \in \cL_k} \left\{ \1_i(j) =1 \right\}\right),
\end{equation}
where
\begin{equation}\label{eq:pop}
\P \left\{ \1_i(s_1) =1\right\} = \P \left(\left\{ s_1 \text{ is active } \right\} \cap (s_1,i) \right)
\leq \P \left\{ (s_1,i)\right\} = p.
\end{equation}
Equation \eqref{eq:pop} is exactly equation \eqref{eq:key} rephrased in another setting.

For any subset of $\{1,...,t \}$, we have
\begin{equation}\label{eq:lessthanpk1}
\P \left( \bigcap_{j \in \cL_k} \left\{ \1_i(j) =1 \right\}\right) \leq \P\left\{ (s_1,i)\cap ... \cap (s_k,i) \right\} = p^k,
\end{equation}
and the inequality \eqref{eq:lessthanpk1} is fulfilled for any choice of $\cL_k$. The number of such lists is obviously smaller than the number of subset of length $k$. Therefore

\begin{equation}
 \P \left\{ \bigcup_{\cL_k \subseteq \{1,...,t\}} \left( \sum_{j \in \cL_k} \1_i(j) =k \cap \sum_{j \notin \cL_k} \1_i(j) \geq 0  \right) \right\}
\leq \P\left\{ \Bin(t,p)  \geq k \right\}.
\end{equation}
That means
\begin{equation}\label{eq:dominator}
\P \left\{ M_i(t) \geq k \right\} \leq  \P\left\{ \Bin(t,p)  \geq k \right\},
\end{equation}
for any $k \leq t$.
\end{proof}

\begin{lem}\label{lem:pi+}
Let
\begin{equation}\label{eq:pi+}
\pi^+(t) = \P \left\{ \Bin(t,p)  \geq \max \left( \Bin(n-1-t,p); 1\right) \right\},
\end{equation}
then
\begin{equation}\label{eq:pileqpi+}
\pi(t) \leq \pi^+(t) \quad \text{for any } t.
\end{equation}
Moreover, for $t \leq A(0)$, the vertices $s \leq t$ are initially active therefore, the probability that a vertex $i$ receives a mark from $s$ is exactly the probability to have an edge between them, thus
\begin{equation}\label{eq:pi=pi+}
\pi(t) = \pi^+(t) \quad \text{ for } t \leq A(0).
\end{equation}
\end{lem}
\begin{proof}[Proof of Lemma \ref{lem:pi+} ]
To begin with, we recall equation \eqref{eq:lessthanpk1}. For any subset $\cL_k \subset \{1,...,t\}$ with $|\cL_k| = k$
\begin{equation}\label{eq:lessthanpk}
\P \left( \bigcap_{j \in \cL_k} \left\{ \1_i(j) =1 \right\}\right) \leq \P\left\{ (s_1,i)\cap ... \cap (s_k,i) \right\} = p^k
\end{equation}
Consider the probability of activation
\begin{align}\label{eq:pi1}
\pi(t)
& = \P \left\{ M_i(t) \geq \max\left( \Bin_i([t+1,n],p) ; 1 \right) \right\}
\nonumber
\\
& = \sum_{k=1}^{t} \P \left( \left\{ M_i(t) \geq k \right\} \cap \left\{ \max\left( \Bin_i([t+1,n],p) ; 1 \right) =k \right\} \right)
\end{align}
By lemma \ref{lem:stocdomin}, using equation \eqref{eq:binn-t-1} and \eqref{eq:dominator} in \eqref{eq:pi1} 
\begin{align}\label{eq:pi2}
\pi(t)
& \leq \sum_{k=1}^{t}  \P \left( \left\{ \Bin_i([1,t],p)  \geq k \right\} \cap \left\{ \max\left( \Bin_i([t+1,n],p) ; 1 \right) =k \right\} \right)
\nonumber
\\
& \leq \P \left\{ \Bin(t,p)  \geq  \max\left( \Bin(n-t-1,p) ; 1\right)  \right\} = \pi^+(t)
\end{align}
\end{proof}
In the proofs, we will use equality \eqref{eq:pi=pi+} with the fact that
\begin{equation}\label{eq:lowerbound}
A\left( A(0)\right) \leq A^*,
\end{equation}
to determine conditions for the supercritical case.
 To prove Theorem \ref{theo:p>>1overn} \ref{theo:p>>1overniii}, we show that by the time the vertices of $\cA(0)$ have been explored, the process has already almost percolated.
In order to find conditions for the process to stay subcritical, we use the inequality \eqref{eq:pileqpi+} and define the random process $\left(S^+(t)\right)_{t \leq n}$ with $S^+(t) \in \Bin\left( n-A(0), \pi^+(t)\right)$.
In the following, we show that $S^+(t)$ stochastically dominates $S(t)$.
\section{Subcritical phase, a useful upper bound}\label{Sbound}
It is simpler to start by proving that the random variable $R(t)$ dominates a certain binomial random variable. It is easy to see that the random variables $K_i(t)$, with $R(t) = \sum K_i(t)$ (see equation \eqref{eq:defrt}) are positively related, see equation \eqref{eq:postivrel} below. The same question is more complicated with the random variables $I_i(t)$ as it depends on whether the connections have been revealed or not (see Figure \ref{graph:gnpact}). We further use that $R(t) + S(t) = n-A(0)$ to transfer the result in terms of $S(t)$ and $S^+(t) \in \Bin\left( n-A(0), \pi^+(t)\right)$.
\begin{lem}\label{lem:larger}
For any $t$ and $k_0 \geq 0$
\begin{equation}\label{eq:larger}
\P \left\{ R(t) \geq k_0\right\} \geq \P \left\{ \Bin(n-A(0),\delta(t)) \geq k_0 \right\}.
\end{equation}
\end{lem}
The proof of Lemma \ref{lem:larger} is kind of the reverse of the proof of Lemma \ref{lem:stocdomin}.
Conversely to Lemma \ref{lem:stocdomin}, in the case of Lemma \ref{lem:larger}, the random variable $R(t)$ dominates the binomial.
The random variables  $K_i$ are positively related. Let $\cL_k$ be some subset of $V \setminus \cA(0)$ of $k$ elements, then if some vertices are inactive, that is $\left\{ \cap_{j \in \cL_k } K_j (t)=1 \right\}$, they tend to keep the other vertices inactive too, that is $\left\{ K_i (t)= 1 \right\}$ and we have
\begin{equation}\label{eq:postivrel}
\P \left\{ \bigcap_{j \in \cL_k} K_j (t)= 1\right\} \geq \prod_{j \in \cL_k} \P \left\{ K_j (t)=1\right\}.
\end{equation}
The inequality \eqref{eq:postivrel} can be derived for 2 random variables, that is $k=2$ and extended to any $k$ by induction.

The inequality was reversed in the proof of Lemma \ref{lem:stocdomin} and we didn't have to worry about the number of combinations. In the case of Lemma \ref{lem:larger}, it is crucial that the number of subsets $\cL_k$ is equal to the number  of combinations of the binomial. This is ensured by the fact that the random variables $K_j(t)$ are exchangeable.
\begin{proof}[Proof of Lemma \ref{lem:larger}]
From the beginning, we have that the relation \eqref{eq:larger} is verified for $k_0=0$ since both probabilities equal 1.

We recall that $R(t) = \sum_{i=1}^{n-A(0)} K_i(t)$ more precisely, we will write
$R_{n-A(0)} = \sum_{i=1}^{n-A(0)} K_i$ to emphasise the dependence on the number of terms we sum up and will omit the indicator of time $t$.
The random variables $K_i$ are exchangeable, therefore
\begin{align*}
\P \left\{ R_{n-A(0)} \geq k_0 \right\}
& = \P \left(\left\{ R_{n-A(0)-k_0} \geq 0\right\} \cap \left\{ K_{n-A(0)-k_0 +1} =1\right\} \cap ... \cap \left\{K_{n-A(0)} =1 \right\}\right) \alpha_{n-A(0),k_0}
\\
&  = \P \left\{ R_{n-A(0)-k_0} \geq 0 \right\} \P \left\{ R_{k_0} = k_0 \right\}\alpha_{n-A(0),k_0},
\end{align*}
where $\alpha_{n-A(0),k_0}$ denotes the number of combinations. 

The random variables $K_i$ are positively related. 
So for any $m$ such that $m\geq 1$
\begin{equation}\label{eq:labelcasen}
\P \left\{ R_{m} = m \right\} \geq \P \left\{ \Bin(m,\delta) = m \right\}.
\end{equation}
Taking $m = n-A(0)$ in the inequality \eqref{eq:labelcasen}, we see that the relation \eqref{eq:larger} is verified for $k = n-A(0)$ too.

Because the indicator functions $K_j$ are exchangeable, the number of combinations $\alpha_{n-A(0),k_0}$ is the same for $\{R_{n-A(0)} \geq k_0\}$ and $\{ \Bin(n-A(0),\delta) \geq k_0\}$
\begin{align*}\label{eq:}
\frac{ \P \left\{ R_{n-A(0)} \geq k_0\right\} }{\P \left\{ \Bin(n-A(0),\delta) \geq k_0 \right\}}
& = \frac{ \P\left( \left\{ R_{n-A(0)-k_0} \geq 0 \right\} \cap \left\{ R_{k_0} = k_0\right\} \right)}{\P \left( \left\{ \Bin(n-A(0)-k_0,\delta) \geq 0 \right\} \cap \left\{ \Bin(k_0, \delta) = k_0 \right\} \right) } \frac{\alpha_{n-A(0),k_0}}{\alpha_{n-A(0),k_0}}.
\end{align*}
The events 
$ \left\{ R_{n-A(0)-k_0} \geq 0 \right\}$ and $\left\{ \Bin(n-A(0)-k_0,\delta) \geq 0 \right\} $ are always fulfilled. Hence
\begin{equation*}
 \frac{ \P\left( \left\{ R_{n-A(0)-k_0} \geq 0 \right\} \cap \left\{ R_{k_0} = k_0\right\} \right)}{\P \left( \left\{ \Bin(n-A(0)-k_0,\delta) \geq 0 \right\} \cap \left\{ \Bin(k_0, \delta) = k_0 \right\} \right) }
  = \frac{ \P \left\{ R_{k_0} = k_0 \right\}}{ \P\left\{ \Bin(k_0, \delta) = k_0 \right\} }.
 \end{equation*}
 Using \eqref{eq:labelcasen} in the case of $k_0$, we find that
 \begin{align*}
 \frac{ \P \left\{ R_{n-A(0)} \geq k_0\right\} }{\P \left\{ \Bin(n-A(0),\delta) \geq k_0 \right\}}
 & = \frac{\P \left\{ R_{k_0} = k_0 \right\} }{\P\left\{ \Bin(k_0, \delta) = k_0 \right\} } \geq 1,
\end{align*}
which proves Lemma \ref{lem:larger}.
\end{proof}
\begin{cor}\label{cor:upperboundst}
The random variable $S(t)$ is stochastically dominated by $\Bin\left(n-A(0), \pi(t) \right)$ 
\begin{equation}\label{eq:upper1}
\P \left\{ S(t) \geq k \right\} \leq \P \left\{ \Bin\left(n-A(0), \pi(t) \right) \geq k\right\}.
\end{equation}
Moreover
\begin{equation}\label{eq:upper2}
\P \left\{ S(t) \geq k \right\} \leq \P \left\{ \Bin\left(n-A(0), \pi^+(t) \right) \geq k\right\} = \P \left\{ S^+(t) \geq k \right\}
\end{equation}
\end{cor}
\begin{proof}[Proof of Corollary \ref{cor:upperboundst}]
We have $n = A(0) + S(t) + R(t)$, so
\begin{align*}
\P \left\{ S(t) \geq k \right\}
& = \P \left\{ n-A(0) - R(t) \geq k \right\}
\\
& = \P \left\{ R(t) \leq n- A(0) -k \right\}
\\
&  \leq \P \left\{ \Bin \left( n-A(0),1 - \pi(t) \right) \leq n-A(0) - k \right\}.
\end{align*}
Since
\begin{equation*}
\P \left\{ \Bin \left( n-A(0),1 - \pi(t) \right) \leq n-A(0) - k \right\}
 = \P \left\{ \Bin \left( n-A(0),\pi(t) \right) \geq k \right\},
\end{equation*}
we deduce that
\begin{equation}
\P \left\{ S(t) \geq k  \right\} \leq \P \left\{ \Bin\left( n-A(0), \pi(t)\right)\geq k \right\},
\end{equation}
which is equation \eqref{eq:upper1}. Equation \eqref{eq:upper2} follows from the fact that $\pi^+(t) \geq \pi(t)$
(see equation \eqref{eq:pileqpi+}.
\end{proof}

\section{The case $p=o\left( \frac{1}{n} \right)$, proof of Theorem \ref{theo:o1overn}}\label{Sproof1}
In the case $p=o\left( \frac{1}{n} \right)$, we are going to prove that the system is subcritical. 
Indeed, there are so few connection that the activation cannot spread along it. 
We use a very crude bound for the probability of a vertex to be activated by using the condition that this vertex needs to receive at least one incoming activation.
\begin{proof}[Proof of Theorem \ref{theo:o1overn}]
We have in general
\begin{align*}
\pi(t) \leq \pi^+(t) & =  \P \left( \left\{ \Bin_i([1,t],p)  \geq \Bin_i([t+1,n],p)\right\} \cap \left\{ \Bin_i([1,t],p)  > 0 \right\} \right) 
\\
& \leq \P \left\{ \Bin_i([1,t],p)  > 0 \right\}.
\end{align*}
Using that $p = o\left( \frac{1}{n} \right) = o\left( \frac{1}{t} \right) $, we derive
\begin{equation*}
\P \left\{ \Bin_i([1,t],p)  > 0 \right\} = 1- \P \left\{ \Bin_i([1,t],p)  = 0 \right\} = 1- \left(1 - tp \etto \right) = tp \etto.
\end{equation*}
Therefore, using Corollary \ref{cor:upperboundst}, the expected number of vertices i.e. $\E\left( S(t)\right)$ that have been activated by time $t$ is bounded from above by
\begin{equation}\label{eq:s+ot}
\E \left(S^+(t) \right) = \left(n-A(0)\right) \pi^+(t) \leq n tp \etto = o(t).
\end{equation}
Using Markov's inequality, we deduce for any $\lambda >0$ that
\begin{equation*}
\lim_{t \to \infty} \P \left\{ \frac{S^+(t)}{t} > \lambda \right\} = 0.
\end{equation*}
Letting $t = (1+\epsilon) A(0)$ and $\lambda = \frac{\epsilon}{1+\epsilon}$, we derive that
\begin{equation*}
\lim_{n \to \infty} \P \left\{ S^+\left((1 + \epsilon)A(0)\right) -  \epsilon A(0) >  0\right\} = 0,
\end{equation*}
implying by domination (see Corollary \ref{cor:upperboundst}) that the process stops before time $t = (1 + \epsilon)A(0)$ for any positive $\epsilon$. Therefore, for $p = o\left( \frac{1}{n}\right)$ and any $A(0)$, we have
\begin{equation*}
\Ax = A(0) \ettop.
\end{equation*}
For $A(0) = O(1)$ then using equation \eqref{eq:s+ot}, we derive $\E\left( S^+\left(A(0)\right)\right) = o(1)$
so $\P \left\{ A^* > A(0) \right\} = o(1)$ and w.h.p, we have $\cA^* = \cA(0)$.
That proves Theorem \ref{theo:o1overn}.
\end{proof}
\section{The case $p = \frac{c}{n}$, proof of Theorem \ref{theo:covern}}\label{Sproof2}

\subsection{Approximation by a Poisson random variable}
In the case of $p = \frac{c}{n}$, it is handy for the computations to approximate the probability $\pi^+(t)$ using the approximation of a binomial by a Poisson random variable.

We use the standard approximation
\begin{equation}\label{eq:dtv}
d_{TV} \left( \Bin(t,p), \Po (tp)\right) < p,
\end{equation}
where $d_{TV}$ denotes the total variation distance. See Theorem 2:M in \cite{BHJ}.

\begin{rem}
The approximation \eqref{eq:dtv} implies that
\begin{equation}\label{eq:approxpoi}
\pi^+(t) =  \P  \left\{ \Po(tp) \geq \max \left( \Po \left((n-t-1)p\right); 1\right) \right\}+ O(p).
\end{equation}
\end{rem}
Indeed, we have using the independence of the links for disjoint sets that
\begin{multline*}
\pi^+(t) =\sum_{k=1}^{n-t-1} \P\left\{\Bin_i([1,t],p)  \geq k \right\} \P \left\{ \Bin (n-t-1,p) = k\right\} 
\\
+  \P\left\{\Bin_i([1,t],p)  \geq 1 \right\} \P \left\{ \Bin (n-t-1,p) = 0 \right\} .
\end{multline*}
We use the approximation by the corresponding Poisson probability to derive
\begin{multline}\label{eq:developoisson}
\pi^+(t) = \sum_{k=1}^{n-t-1}\left( \P \left\{ \Po(tp) \geq k \right\} + O(p) \right) \left( \P \left\{ \Po \left((n-t-1)p\right) = k\right\} +O(p)\right) 
\\
+ \left( \P\left\{\Po(tp) \geq 1 \right\} + O(p) \right)  \left( \P \left\{ \Po \left((n-t-1)p\right) = 0 \right\} + O(p) \right)
\end{multline}

The lower term in equation \eqref{eq:developoisson} is
$\P\left\{\Po(tp) \geq 1 \right\} \P \left\{ \Po \left((n-t-1)p\right) = 0 \right\} + O(p)$.

The upper term in equation \eqref{eq:developoisson} can be developed into
\begin{multline}\label{eq:popo}
\sum_{k=1}^{n-t-1} \P \left\{ \Po(tp) \geq k \right\}  \P \left\{ \Po \left((n-t-1)p\right) = k\right\}
\\
+ O(p) \sum_{k=1}^{n-t-1} \P \left\{ \Po \left((n-t-1)p\right) = k\right\}
+ O(p) \sum_{k=1}^{n-t-1} \P \left\{ \Po(tp) \geq k \right\} 
+ O(p^2) \sum_{k=1}^{n-t-1} 1.
\end{multline}
We bound the terms on the lower line of equation \eqref{eq:popo}.
For the first term, we use the bound $\sum_{k=1}^{n-t-1} \P \left\{ \Po \left((n-t-1)p\right) = k\right\} \leq 1$.

For the second term, we have $\sum_{k=1}^{n-t-1} \P \left\{ \Po(tp) \geq k \right\} \leq \E \left(\Po (pt)\right) = pt = O(1)$ since $t\leq n$ and $p=\frac{c}{n}$.

For the last term we obviously have $ \sum_{k=1}^{n-t-1} 1=n-t-1$.

Inserting these bounds into \eqref{eq:developoisson}, we derive equation \eqref{eq:approxpoi}.

Computations of the relation \eqref{eq:approxpoi} give
\begin{align}\label{eq:pitpoisson}
\pi^+(t) 
& =  \sum_{k=1}^{t} \frac{(pt)^k}{k!} e^{-pt}  \sum_{j=0}^{k}\frac{\left((n-t-1)p\right)^j}{j!} e^{-(n-t-1)p} +O(p)
\nonumber
\\
\pi^+(t) & = e^{-(n-1)p} \sum_{k=1}^{t}  \frac{(pt)^k}{k!}   \sum_{j=0}^{k} \frac{\left((n-t-1)p\right)^j}{j!} +O(p).
\end{align}
The random variables $\Bin_i([1,t],p) $ and $\Bin_i([t+1,n],p)$ determine the number of links a certain vertex has with two disjoint set of vertices. By independence of the connections, the random variables $\Bin_i([1,t],p) $ and $\Bin_i([t+1,n],p)$ are independent.
The random variables $ \Po \left((n-t-1)p\right)$ and $ \Po \left((n-t-1)p\right)$ associated with their respective binomials are independent as well.

\subsection{Subcritical case,  $p = \frac{c}{n}$ and $A(0) = o(n)$}
\begin{proof}[Proof of Theorem \ref{theo:covern} \ref{theo:coverni}]
We consider the case $p = \frac{c}{n}$ and $A(0) = o(n)$. We study the process of activation along time $t$.
Eventually, $t$ will be a multiple of $A(0)$ so we assume throughout the calculations that $t=o(n)$.
 
We split the probability $\pi^+(t)$ 
into two terms, $k=1$ and $k\geq 2$
\begin{multline}\label{eq:split}
\pi^+(t) = \P \left( \left\{ \Bin(t,p)  = 1 \right\} \cap \left\{ \Bin(n-t-1,p) \leq 1\right\} \right)
\\
+ \P \left( \left\{ \Bin(t,p)  \geq \Bin(n-t-1,p) \right\} \cap \left\{ \Bin(t,p)  \geq 2 \right\} \right) + O(p).
\end{multline}
Using the approximation \eqref{eq:pitpoisson}, 
we deduce for each term of \eqref{eq:split} that for $t=o(n)$
\begin{equation*}
\P \left( \left\{ \Bin(t,p)  \geq \Bin(n-t-1,p) \right\} \cap \left\{ \Bin(t,p)  \geq 2 \right\} \right)
= e^p e^{-np} O(p^2t^2) + O(p),
\end{equation*}
and
\begin{equation*}
\P \left( \left\{ \Bin(t,p)  = 1 \right\} \cap \left\{ \Bin(n-t-1,p) \leq 1\right\} \right)
= e^{p} e^{-np} pt \left( 1 + p(n-t-1)\right) + O(p).
\end{equation*}
Therefore, we have 
\begin{equation}\label{eq:boundbyexp}
\pi^+(t) = (1+pn) p e^{-np} t \etto + O(p).
\end{equation}
To prove that the process does not percolate, we use again that the random variable $S(t)$ is stochastically dominated by $S^+(t) \in \Bin\left( n-A(0), \pi^+(t)\right)$.

Recall $t = o(n)$ such that $pt = o(1)$ since $p = \frac{c}{n}$. 
Using the relation \eqref{eq:boundbyexp},
we bound the expectation of the random variable $S^+(t) \in \Bin \left( n-A(0), \pi^+(t)\right)$ by 
\begin{align*}
\E \left( S^+(t)\right) 
& = \left( n- A(0) \right) \pi^+ (t)
\\
& \leq n \pi^+ (t)
\\
& \leq (1+ pn) np e^{-np} t \etto +O(1)= g(c) t \etto.
\end{align*}
where $g(c) = (1+c) c e^{-c}$. Notice that the function $g(c)$ has a maximum 
$(2 + \sqrt{5})e^{-\frac{1+\sqrt{5}}{2}} < 0.84  < 1$ at $c = \frac{1+ \sqrt{5}}{2}$. Therefore, for small $\epsilon$, we will always have $g(c)(1+\epsilon) <1$.
We have for some $\epsilon>0$ and for sufficiently large $n$
\begin{align}\label{eq:84}
\Var \left( \Bin \left( n-A(0), \pi(t)\right) \right) 
& \leq \E \left( \Bin \left( n-A(0), \pi(t)\right) \right) 
\nonumber
\\
& \leq \E \left( \Bin \left( n-A(0), \pi^+(t)\right) \right)
 \leq g(c) t (1+\epsilon)
\end{align}
Under the same conditions as equation \eqref{eq:84}, the probability of survival is
\begin{align*}
\P\left\{ A^* > t \right\} 
& \leq \P \left\{ A(t) >t \right\} 
\\
& = \P \left\{ A(0) + S(t) > t \right\} = \P \left\{ S(t) >  t- A(0)\right\}
\\
& \leq \P \left\{S^+(t) > t - A(0)\right\}
\\
& \leq \P \left\{ S^+(t) - \E \left( S^+(t) \right) > t - A(0) - g(c) t (1+\epsilon) \right\} 
\\ 
& \leq \P \left\{ S^+(t) - \E \left( S^+(t) \right) > \left( 1- g(c) (1+\epsilon) \right) t -A(0) \right\}
\end{align*}
where the second inequality follows from the stochastic domination of Corollary \ref{cor:upperboundst} and the third inequality from \eqref{eq:84}.

Use Chebyshev's inequality with $t = \frac{1+\epsilon}{1-g(c)(1+\epsilon)} A(0)$. We find
\begin{align*}
\P \left\{ A(t) >t \right\} 
& \leq \frac{\Var \left( \Bin \left( n-A(0), \pi(t)\right) \right)}{\left( \epsilon  A(0)\right)^2} 
\\
& \leq \frac{\E \left( \Bin \left( n-A(0), \pi(t)\right)\right)}{\left( \epsilon  A(0)\right)^2}
\\
& \leq \frac{\E \left( \Bin \left( n-A(0), \pi^+(t)\right)\right)}{\left( \epsilon  A(0)\right)^2} 
\\
& \leq \frac{g(c) A(0)}{\left( \epsilon  A(0)\right)^2}
\\
& \leq \frac{g(c)}{\epsilon^2 A(0)} \to 0 \qquad \text{as } n \to \infty,
\end{align*}
if $A(0) \to \infty$ as $n \to \infty$.  That means
\begin{equation*}
\lim_{n \to \infty} \P \left\{ A^* > \frac{1+\epsilon}{1-g(c)(1+\epsilon)} A(0) \right\} = 0.
\end{equation*}
Since the variable $A^*$ is an monotone increasing in $A(0)$, by boundedness, we derive for $A(0) = O(1)$ that $A^* = o_p\left( w(n)\right)$ for any $w(n) \to \infty$.

That implies immediately that if $A(0) = o(n)$ then 
\begin{equation*}
A^* = o_p(n).
\end{equation*}

\end{proof}
\subsection{Approximation of $S^+(t) = \Bin\left( n-A(0),\pi^+(t)\right)$ by its mean}\label{Sapprox}
This part is necessary in the case of $p=\frac{c}{n}$ and $A(0) = \theta n$ because we approximate the sequence of random variables $S^+(t) \in \Bin\left( n-A(0), \pi^+(t)\right)$ by the expectation $\E\left( S^+(t)\right)$. 
The Glivenko-Cantelli lemma gives a uniform bound on the approximation. That gives us the stopping time for the process $A^+(t) = A(0) + S^+(t)$ which we will denote $T^+$ 
and then we derive an upper bound for $A^* = T$ (see equation \eqref{at}). 

The random variable $S^+(t)$ is a binomial distribution, so for every $t = t(n)$, we have
\begin{equation}\label{eq:preglivenko}
S^+(t) = \E\left(S^+(t)\right) +o_p(n) = (n-A(0)) \pi^+(t) + o_p(n)
\end{equation}
and by the Glivenko-Cantelli lemma \cite{K}, this holds uniformly so
\begin{equation}\label{eq:glivenko}
\sup_{t \geq 0} \Big| S^+(t) - \E(S^+(t)) \Big| =o_p(n).
\end{equation}
For the expected value of $S^+(t)$, we find, using the approximation of $\pi^+(t)$ in equation \eqref{eq:pitpoisson}
\begin{align*}
\E \left( S^+(t) \right) 
& = \left( n-A(0)\right) \pi^+(t)
\\
& = \left( 1 - \theta \right) n \pi^+(t)
\\
& = n \left( 1 - \theta \right) e^p e^{-c} \sum_{k=1}^{\lfloor xn \rfloor} \frac{(cx)^k}{k!} \sum_{j=0}^k \frac{\left((1-x)c\right)^j}{j!} + O(1).
\end{align*}
Consider now $\E \left( A^+(t)\right) -t$ with $t = xn$,
\begin{align*}
\E \left(A^+(t) \right) - t 
& = A(0) + \E \left( S^+(t) \right) -t
\\
& = \theta n - xn + n \left( 1 - \theta \right) e^p e^{-c} \sum_{k=1}^{\lfloor xn \rfloor} \frac{(cx)^k}{k!} \sum_{j=0}^k \frac{\left((1-x)c\right)^j}{j!} + O(1)
\\
& = n \left( \theta  - x +  \left( 1 - \theta \right) e^p e^{-c} \sum_{k=1}^{\lfloor xn \rfloor} \frac{(cx)^k}{k!} \sum_{j=0}^k \frac{\left((1-x)c\right)^j}{j!}\right)  + O(1).
\end{align*}
Let
\begin{equation*}
f_{c,\theta}(x) = \theta  - x +  \left( 1 - \theta \right) e^p e^{-c} \sum_{k=1}^{\lfloor xn \rfloor} \frac{(cx)^k}{k!} \sum_{j=0}^k \frac{\left((1-x)c\right)^j}{j!} ,
\end{equation*}
so that we have
\begin{equation}\label{eq:expect}
\E \left( A^+(t) \right) -t = n f_{c,\theta}(x) +O(1).
\end{equation}
An approximation of the stopping time of the process $A^+(t)$ denoted $T^+$ is given by
\begin{equation}\label{eq:defx0}
x_0 = \inf \{x \geq \theta,  f_{c,\theta}(x)< 0\}.
\end{equation}
This is the smallest root $x_0(c,\theta) \geq \theta$ of $f_{c,\theta}(x) =0$ for given $c$ and $\theta$ such that 
\begin{equation}\label{eq:conduit}
\begin{cases}
&  f_{c,\theta}(x) \geq 0 \quad \text{for } x \leq x_0
\\
\exists \upsilon > 0 \text{ such that} & f_{c,\theta}(x) < 0 \quad \text{for } x \in( x_0,x_0+\upsilon).
\end{cases}
\end{equation}
The condition \eqref{eq:conduit} is to ensure that the function $f_{c,\theta}(x)$ changes sign at $x_0$ and avoid points for which $f_{c,\theta}(x) =  f'_{c,\theta}(x) =0 $ (see remark \ref{rem:nasty}) so that $x_0$ is a double root with $f_{c,\theta}(x) \geq 0$ on a boundary of $x_0$.

Finally, notice that the function $f_{c,\theta}(x)$ is continuous on $[0,1]$ and positive for $x<x_0$.

We give in the following some basic properties to $f_{c,\theta}(x)$ that immediately translates to $\\E\left(S^+(t)\right)$ and further to the process $S^+(t)$ using either \eqref{eq:preglivenko} for concentration results point wise or the Glivenko-Cantelli Lemma for concentration results needed on an interval. The first proposition shows that in the case when $p = \frac{c}{n}$, the activation cannot spread to almost all the graph.

\begin{prop}\label{prop:existence}
Let $p = \frac{c}{n}$, $c>0$. For the process starting from time $A(0) = \theta n$, $\theta <1$, there exists a stopping time $T = A^* = \theta^* n + o_p(n)$ with $\theta^* \leq x_0 <1$.
\end{prop}

\begin{proof}[Proof of proposition \ref{prop:existence}]
We have
\begin{align*}
f_{c,\theta}(1)
& = \theta -1 + (1 - \theta) e^p e^{-c}  \sum_{k=1}^{ n } \frac{(c)^k}{k!} \sum_{j=0}^k \frac{0^j}{j!}
\\
& \leq (1- \theta) \left( e^p e^{-c} \left( e^c -1 \right) - 1 \right).
\end{align*}
We have $e^p = e^{\frac{c}{n}} = 1+\frac{c}{n} \etto$. Therefore, for any $0<\epsilon <e^{-c} $, there exists $ n_{\epsilon}$ such that for any $n \geq n_{\epsilon}$
\begin{equation}\label{eq:<1}
 f_{c,\theta}(1) \leq  (1-\theta) \left((1 + \epsilon) e^{-c} \left( e^c -1\right) -1 \right) = (1-\theta) \left( \epsilon -(1+\epsilon) e^{-c} \right) < 0,
\end{equation}
and the inequality \eqref{eq:<1} holds for any $\theta < 1$.
Along with $f_{c,\theta}(0) = \theta > 0$, that implies that there is at least one solution $<1$ to the equation $f_{c,\theta}(x) =  0$.
Let $x_0$ be defined by \eqref{eq:defx0}.
Clearly, by \eqref{eq:conduit}, for $x = x_0 + \gamma$, with $\gamma < \upsilon$ we have $f_{c,\theta} = - \lambda <0$.
That means for $t=xn$ that
\begin{equation*}
\E \left( A^+(t)\right) - t = A(0) + \E \left( S^+(t)\right) - t = -\lambda n.
\end{equation*}
Using equation \eqref{eq:glivenko}, we derive that
\begin{align*}
A^+(t) - t & = A^+(t) - \E \left( A^+(t)\right) +\E \left( A^+(t)\right) - t
\\
& = o_p(n) - \lambda n.
\end{align*}
Therefore, for $t =xn$, $x = x_0 + \gamma$
\begin{equation*}
\lim_{n \to \infty} \P \left\{ A^+(t) - t > 0 \right\} = 0
\end{equation*}
and this holds for any $\gamma < \upsilon$ thus we have
\begin{equation*}
T^+ \leq x_0 n + o_p(n) \Leftrightarrow T^+ = \theta^+ n \text{ with } \theta^+ \leq x_0.
\end{equation*}
Using the boundedness result of Corollary \ref{cor:upperboundst}, it follows that the process $\left(A(t) \right)_{t \leq n}$ has a stopping time $T = \Ax = \theta^* n +o_p(n)$ with $\theta^*\leq \theta^+ \leq x_0 <1$.
\end{proof}
\begin{rem}\label{rem:nasty}
If we have
$f_{c, \theta} (x_1) = 0$ but  $f_{c, \theta} (x)$ 
does not change sign around $x_1$, as it is required in \eqref{eq:conduit}, then, using simply the Glivenko-Cantelli Lemma, see relation \eqref{eq:glivenko}, we cannot conclude that the process of activation stops or not.
We may have a similar behaviour as in Theorem 5.5 of \cite{JLTV}.
\end{rem}
\begin{rem}\label{rem:allright}
In the case when $x_0$ is the smallest root then, we have $f_{c,\theta}(x) > 0$ on $(0,x_0)$ and one can prove using the Glivenko-Cantelli Lemma that w.h.p $A(t) - t >0$ for all $t=xn$ with $x <x_0$. Hence, one can derive that $T^+ =  x_0 n + o_p(n)$ and $\theta^+ = x_0$.
\end{rem}
\begin{prop}\label{prop:spread}
Let $p = \frac{c}{n}$, $c>0$ and $A(0) = \theta n$, $0< \theta <1$ then the activation spreads to a significantly larger part of the graph
\begin{equation}\label{eq:spread}
A^* = \theta^* n \quad \text{with} \quad \theta^* > \theta.
\end{equation}
\end{prop}
\begin{proof}[Proof of proposition \ref{prop:spread}]
Let us first remark that for $x< \theta$ we have $f_{c,\theta} (x) >0$. Indeed as a first approximation, we have
\begin{equation*}
\frac{\E\left( A(t)\right)}{n} - x \geq \frac{A(0)}{n} - x = \theta - x > 0.
\end{equation*}
Secondly, we use the fact that $A\left( A(0) \right) \leq A^*$ with
\begin{align*}
A\left(A(0)\right)
& = \E \left(A\left(A(0)\right)\right) + o_p(n) \\
& = \E \left(A^+\left(A(0)\right)\right) + o_p(n)\\
& = f_{c,\theta} (\theta) n + A(0)+ o_p(n),
\end{align*}
where the second inequality follows from the fact that $\pi(t) = \pi^+(t)$ for $t \leq A(0) = \theta n$ and the third equality follows from \eqref{eq:expect}.
Let us compute $f_{c,\theta} (\theta)$,
\begin{align*}
f_{c,\theta}(\theta)
& = \theta - \theta + (1 - \theta) e^p e^{-c} \sum_{k=1}^{\theta n} \frac{(c\theta) ^k}{k!} \sum_{j=0}^k \frac{\left( c (1-\theta)\right)^j}{j!}
\\
& = (1 - \theta) e^p e^{-c} \sum_{k=1}^{\theta n} \frac{(c\theta) ^k}{k!} \sum_{j=0}^k \frac{\left( c (1-\theta)\right)^j}{j!} >0.
\end{align*}
That implies that there exists $\lambda >0$ such that  for $n$ large enough $\frac{A\left(A(0)\right)}{n} =  \theta + f_{c,\theta}(\theta) +o_p(1) \geq \left( \theta + \lambda \right)$.
Thus $A^* = \theta^* n + o_p(n)$ with $\theta^* \geq \theta_1 > \theta$.
\end{proof}
We have the necessary results to prove Theorem  \ref{theo:covern}.
\begin{proof}[Proof of Theorem \ref{theo:covern}]
Propositions \ref{prop:existence} implies that $\theta^*<1$ w.h.p, and \ref{prop:spread} implies that $A^* = \theta^* n + o_p(n)$ with $\theta^*> \theta$.

Moreover, from Proposition \ref{prop:existence} , we derive that $\theta^* \leq x_0$ with $x_0$ defined by \eqref{eq:defx0theo}.
That proves Theorem  \ref{theo:covern} \ref{theo:covernii}.
\end{proof}
We studied in the \refS{Sproof1} the case $p = o(\frac{1}{n})$. It is possible to recover some of these results using 
\begin{prop}\label{prop:extra}
\begin{align*}
\lim_{c \to 0} f_{c,\theta}(x) 
& = \lim_{c \to 0}  \theta - x + (1- \theta) e^p e^{-c} \sum_{k=1}^{\lfloor xn \rfloor} \frac{(cx) ^k}{k!} \sum_{j=0}^k \frac{\left( c (1-x)\right)^j}{j!}
\\
& = \theta -x,
\end{align*}
thus we have $f_{c,\theta}(x) <0$ for $x > \theta$.
\end{prop}
One can deduce from Proposition \ref{prop:extra} using the same technique as in the proofs of the proposition \ref{prop:existence} and \ref{prop:spread} that for $p = o \left( \frac{1}{n}\right)$ and $A(0) = \theta n$ then $A^* = A(0) \ettop$ which was proved in \refS{Sproof1}.
\section{The case $ \frac{1}{n} \ll p \leq 1$, proof of Theorem \ref{theo:p>>1overn}}\label{Sproof3}
\subsection{The sub case $A(0) = o(n)$,  proof of  \ref{theo:p>>1overni}}
In the following, we prove that if $A(0)=o(n)$ and $p \gg \frac{1}{n}$ the process is subcritical and thus the final set of active vertices has a size $A^* = o_p(n)$.
\begin{proof}[Proof of Theorem \ref{theo:p>>1overn} \ref{theo:p>>1overni}]
We will consider $t=o(n)$ along the proof.
As in the proof of Theorem \ref{theo:o1overn}, we will use the fact that  $A(t)$ is stochastically dominated by $A^+(t)$ and we will show that for any $\epsilon > 0 $, $A^+\left((1+\epsilon)A(0)\right) - (1+\epsilon)A(0) \leq 0$ w.h.p.

We recall that the process $A^+(t)$ is defined by $A^+(t) = A(0) + S^+(t)$ where $S^+(t) \in \Bin\left( n-A(0), \pi^+(t)\right)$ and $\pi^+(t) = \P\left\{ \Bin(t,p)  \geq \max\left( \Bin(t,p),1 \right)\right\}$.
We start by splitting $\pi^+(t)$ in two
\begin{align}\label{eq:debase}
\pi^{+}(t) 
 & = \P \Biggl( \left\{ \Bin(t,p)  \geq \max \left(\Bin \left(n-t-1,p\right);1\right) \right\} 
 \\
&\qquad \qquad  \qquad \cap \left( \left\{ \Bin(t,p)  \geq \frac{1}{4}np\right\} \cup \left\{ \Bin(t,p)  \leq \frac{1}{4}np\right\} \right) \Biggr)
\nonumber
\\
& = \P \Biggl( \left\{ \Bin(t,p)  \geq \max \left(\Bin \left(n-t-1,p\right);1\right) \right\}
\\
 &\qquad \qquad  \qquad \cap \left( \left\{ \Bin(t,p)  \geq \frac{1}{4}np\right\} \cup \left\{ \Bin(n-t-1,p)  \leq \frac{1}{4}np\right\} \right) \Biggr)
\nonumber
\\
& \leq  \P \left( \left\{ \Bin(t,p)  \geq \max \left(\Bin \left(n-t-1,p\right);1\right) \right\} \cap \left\{ \Bin(t,p)  \geq \frac{1}{4}np\right\} \right)
\nonumber
\\
& \qquad + \P \left( \left\{ \Bin(t,p)  \geq \max \left(\Bin \left(n-t-1,p\right);1\right) \right\} \cap \left\{ \Bin(n-t-1,p)  \leq \frac{1}{4}np\right\} \right)
\\
& \leq  \P \left( \left\{ \Bin(t,p)  \geq \max \left(\Bin \left(n-t-1,p\right);1\right) \right\} \cap \left\{ \Bin(t,p)  \geq \frac{1}{4}np\right\} \right)
\nonumber
\end{align}
We use Theorem 2.1 from \cite{JLR} which we recall here. Let $X$ be a binomial random variable
then for $z>0$
\begin{equation}\label{eq:largedev1}
\P \left\{ X \geq \E X + z\right\} \leq \exp \left( - \frac{z^2}{2 \left(\E X + \frac{z}{3} \right)} \right)  
\end{equation}
and
\begin{equation}\label{eq:largedev2}
\P \left\{ X \leq \E X - z \right\} \leq \exp \left( - \frac{z^2}{2 \E X } \right) .
\end{equation}
We have
\begin{align*}
\P\left\{ \Bin(t,p)  \geq \frac{1}{4}np\right\}
& = \P\left\{ \Bin(t,p)  \geq tp + \left( \frac{1}{4}np - tp \right)\right\}
\\
& \leq \exp \left( - \frac{\left(\frac{1}{4}np - tp\right)^2}{2 \left( tp + \frac{1}{3}\left( \frac{1}{4}np - tp \right) \right)}  \right).
\end{align*}
Use that $t = o(n)$ to derive
\begin{align}\label{eq:bound1}
\P\left\{ \Bin(t,p)  \geq \frac{1}{4}np\right\} 
& \leq \exp \left( - \frac{1}{3}np \right).
\end{align}
We also have
\begin{align*}
\P \left\{ \Bin(n-t-1,p) \leq \frac{1}{4}np\right\} 
& = \P \left\{ \Bin(n-t-1,p) \leq (n-t-1)p - \left( (n-t-1)p - \frac{1}{4}np \right) \right\} 
\\
& \leq \exp \left( - \frac{\left( (n-t-1)p - \frac{1}{4}np \right)^2}{2 (n-t-1)p} \right).
\end{align*}
Use that $t = o(n)$ to derive
\begin{align}\label{eq:bound2}
\P \left\{ \Bin(n-t-1,p) \leq \frac{1}{4}np\right\} 
& \leq \exp \left(- \frac{1}{3} np \right).
\end{align}
The bounds \eqref{eq:bound1} and \eqref{eq:bound2} imply
\begin{equation}\label{eq:boundtotal}
\pi^+(t) \leq  2 \exp \left(- \frac{1}{3} np \right). 
\end{equation}

Since $A(0) = o(n)$ and $p \gg \frac{1}{n}$, we consider $t = o(n)$ we find, using Corollary \ref{cor:upperboundst} and Markov's inequality, that
\begin{align}\label{eq:procedure}
\P \{A^* >t\}
&\leq \P \left\{ A^+( t ) > t \right\}
\nonumber
\\
& =  \P \left\{ S^+(t) > t-A(0) \right\}
\nonumber
\\
& = \P \left\{ \Bin\left(n-A(0), \pi^+(t)\right) >  t-A(0)   \right\}
\nonumber
\\
& \leq \frac{((n-A(0))\pi^+(t)}{t- A(0)}
\nonumber
\\
& \leq \frac{2n \exp\left(-\frac{1}{3}np\right)}{t- A(0)}.
\end{align}
We consider two cases
\begin{enumerate}
\item If 
\begin{equation}\label{eq:a0large}
A(0) \gg n \exp \left(-\frac{1}{3} np \right),
\end{equation}
then take $t = (1+\epsilon)A(0)$  and use Corollary \ref{cor:upperboundst} to derive
\begin{align*}
\P \left\{ A^* > (1+ \epsilon) A(0)\right\} 
& \leq \P \left\{ A^+\left( (1+ \epsilon) A(0)\right) > (1+ \epsilon) A(0) \right\}
\\
& \leq \frac{2n \exp \left( -\frac{1}{3} np \right)}{\epsilon A(0)} \to 0 \quad \text{ as } n \to \infty
\end{align*}
That implies that  $\Ax =A(0) \ettop = o_p(n)$.

\item 
%
In this case $A(0) \leq K n \exp(-\frac{1}{3} np)$ for a constant $K$. For any $\alpha > 0$ choose a constant $C_\alpha > \frac{2+K\alpha}{\alpha}$. Then
\begin {align}
 \P\left(A^\ast > C_\alpha n \exp(-\frac{1}{3} np)\right) &= \P\left(A(C_\alpha n \exp(-\frac{1}{3} np)) > C_\alpha n \exp(-\frac{1}{3} np)\right) \\
 &\leq \frac{2 n \exp(-\frac{1}{3} np)}{C_\alpha n \exp(-\frac{1}{3} np) - A(0)} \\
 &\leq \frac{2}{C_\alpha - K} < \alpha.
\end {align}
Thus, $A^\ast = O_p(n \exp(-\frac{1}{3} np))$. We recall that since $p\gg \frac{1}{n}$, $np \to \infty$ as $n\to \infty$. Therefore, we have shown that in this case the activation does not spread to a finite proportion of the graph.

\end{enumerate}
That proves Theorem \ref{theo:p>>1overn} \ref{theo:p>>1overni}.
\end{proof}
\begin{proof}[Proof of Theorem \ref{theo:p>>1overn} \ref{theo:p>>1overnii}]
We consider the case $A(0) = \theta n$, $\theta < \frac{1}{2}$. We use that $A(t)$ is stochastically dominated by $A^+(t)$ and prove that $\P \left\{ A^+\left( (1+ \varepsilon)A(0)\right) > (1+\varepsilon) A(0) \right\} = o(1)$.

Let $t=xn$, we have similarly to 
 \eqref{eq:boundtotal}
 \begin{align}
 \pi^+(t) 
 & \leq \P \left\{ \Bin\left( n-t-1,p\right) \leq \frac{1}{2} np \right\} + \P \left\{ \Bin\left( t,p\right) \geq \frac{1}{2} np \right\}
 \nonumber
 \\
 & \leq \P \left\{ \Bin\left( n-t-1,p\right) \leq (n-t-1)p - \left( (n-t-1)p - \frac{1}{2} n p \right)\right\} 
 \nonumber
 \\
 & \qquad
+ \P \left\{ \Bin\left( t,p\right) \geq tp + \left( \frac{1}{2} n - t\right) p \right\}.
 \end{align}
 Using the inequalities \eqref{eq:largedev1} and \eqref{eq:largedev2}, we bound
 \begin{equation*}
 \pi^+(t) 
  \leq \exp\left( - \frac{\left( (n-t-1)p - \frac{1}{2}np\right)^2}{2(n-t-1)p} \right)
 + \exp\left( - \frac{\left(\left(\frac{1}{2}n -t\right)p \right)^2}{2 \left( tp + \frac{\frac{1}{2}n-t}{3}p\right)} \right).
 \end{equation*}
 For any small $\lambda >0$ then for $n$ sufficiently large, we have
 \begin{equation*}
 \pi^+(t)  \leq \exp \left( -  (1-\epsilon)\frac{\left( \frac{1}{2} n-t \right)^2}{2n} p \right) + 
 \exp \left( - \frac{\left( \frac{1}{2} n-t \right)^2}{2n} p \right) .
 \end{equation*}
Let $\omega(n) \to \infty$. Then we have uniformly for any  $t< \frac{1}{2}n - \sqrt{\frac{n}{p}}\omega(n)$, $\pi^+(t) = o(1)$ and more precisely, we have
\begin{equation*}
\E \left( S^+(t)\right) = \E \left( \Bin(n-A(0), \pi^+(t)) \right) 
= o(n).
\end{equation*}
We repeat the same procedure as in equation \eqref{eq:procedure} to derive that for any $0 < \epsilon <\frac{1}{2} - \theta$
\begin{align*}
\P \left\{ A^* > (1+\varepsilon) A(0)  \right\} \leq 
\P \left\{ A^+\left( (1+ \varepsilon)A(0)\right) > (1+\varepsilon) A(0) \right\} = o(1).
\end{align*}
By corollary \ref{cor:upperboundst}, we have that
 if $A(0) = \theta n$, $\theta <\frac{1}{2}$ and  $\frac{1}{n} \ll p \ll 1$ then 
 \begin{equation*}
 A^* = A(0) + o_p(n).
 \end{equation*}
That proves Theorem \ref{theo:p>>1overn} \ref{theo:p>>1overnii}.
\end{proof}
\begin{proof}[Proof of Theorem \ref{theo:p>>1overn} \ref{theo:p>>1overniii}]
In this proof, we show that after exploring the $A(0)$ vertices initially set as active, the set of vertices $\cR (t)= V \setminus \cA(t)$ has w.h.p. a size of order $o(n)$. Let us recall that $|\cR(t)| = R(t) = \sum_{i=1}^{n-A(0)} K_i(t)$ (see equation \eqref{eq:defrt}), where $K_i(t) \in \Be \left( \delta(t)\right)$ with $\delta (t) = 1 - \pi(t)$.
We consider the case $A(0)= \frac{1}{2} n + \omega (n) \sqrt{\frac{n}{p}}$. Recall that for $t \leq A(0)$ then $\pi(t) = \P \left\{ \Bin_i\left([1,t],p \right) \geq \max \left(\Bin_i([t+1,n],p),1 \right)  \right\} = \pi^+(t)$, where the random variables $ \Bin_i\left([1,t],p \right)$ and $\Bin_i([t+1,n],p)$ are independent as they represent links to disjoint set of vertices.
Let $\frac{1}{2}n< t \leq A(0)$ then the probability that a vertex of $V \setminus \cA(0)$ remains inactive at time $t$ is bounded by 
\begin{align*}
\delta(t) = 1- \pi(t)
& = \P \left( \left\{  \Bin_i([1,t],p)  < \Bin_i([t+1,n],p) \right\} \cup  \left\{ \Bin_i([1,t],p) = 0 \right\} \right)
\\
& \leq \P \left\{ \Bin_i([1,t],p)  \leq \Bin_i([t+1,n],p) \right\}
\\
& \leq \P \left\{ \Bin_i([1,t],p)  \leq \frac{1}{2}np\right\} + \P \left\{\Bin_i([t+1,n],p) \geq \frac{1}{2}np \right\}
\\
& \leq \P \left\{ \Bin_i([1,t],p)  \leq tp - \left(tp -  \frac{1}{2}np \right)\right\} 
\\
& \qquad \qquad+ \P \left\{\Bin_i([t+1,n],p) \geq (n-t-1)p + \left( \frac{1}{2}np - (n-t-1)p\right)\right\}.
\end{align*}
 Using the inequalities \eqref{eq:largedev1} and \eqref{eq:largedev2}, we bound
 \begin{align}\label{eq:borndeltat}
 \delta(t) 
& \leq  \exp \left( - \frac{\left(tp- \frac{1}{2}np\right)^2}{2tp}\right) + \exp \left( - \frac{\left(\frac{1}{2}np -(n-t-1)p \right)^2}{2 \left( (n-t-1)p + \frac{\frac{1}{2}np - (n-t-1)p}{3}\right)} \right)
\\
& \leq 2 \exp \left( - \frac{\left( \frac{1}{2}n-t\right)^2}{n} p \right) .
\end{align}
Use the bound \eqref{eq:borndeltat} for $t = \frac{1}{2}n + \sqrt{\frac{n}{p}} \omega (n)$ where $\lim_{n \to \infty} \omega(n) = + \infty$
\begin{equation}
\delta (t) \leq 2 \exp (-\frac{1}{2} \omega^2(n)).
\end{equation}
Therefore, we can bound the expectation of $R(t)$ by
\begin{align}
\E \left( R(t) \right) 
& = \left( n - A(0) \right) \delta(t)
\nonumber
\\
& \leq 2 n \exp \left(-\frac{1}{2} \omega^2(n)\right).
\end{align}
Therefore, we have
\begin{equation}
\E \left( R\left( A(0)\right) \right)  = o(n).
\end{equation}
For $t=A(0)$, we have $R(t) = \op(n)$ and therefore since $A^* \geq A\left( A(0) \right)$
\begin{equation*}
A^* = n - o_p(n).
\end{equation*}
\end{proof}
\section{Conclusion}
In the article, we treated the problem of majority bootstrap percolation on the random graph $\gnp$.
We showed that the process is always subcritical in the case $p=o\left(\frac{1}{n}\right)$.

For a given $p \gg \frac{1}{n}$, we could determine in Theorem \ref{theo:p>>1overn}, the threshold for majority bootstrap percolation, $A_c = \theta n \ettop$ with $\theta =\frac{1}{2}$. 

The upper bound for $A_c$ is actually sharper. We have that if
\begin{equation}
\lim_{n \to \infty} \frac{A(0) -\frac{1}{2}n}{\sqrt{\frac{n}{p}}} = + \infty,
\end{equation}
then
\begin{equation}
A^* = n - o_p (n).
\end{equation}
We believe that $\sqrt{\frac{n}{p}}$ is the right range for the phase transition around the value $A_c=\frac{1}{2}n$. 

Our computation of the lower bound only used that the variable $S(t)$ was stochastically dominated by a random variable $S^+(t)$. A better knowledge of the probability of receiving a mark at time $s$, denoted $p_s$ would bring better results in that direction. 
In order to perform a better lower bound, one needs to consider the behaviour of the process after the round of activation from the vertices of $\cA(0)$ and therefore introduce computations using $p_s$.

It is an open problem whether for $A(0) = \frac{1}{2}n + x \sqrt{\frac{n}{p}}$ for some $-\infty<x<+\infty$ then the graph percolates with a positive probability $\phi$ and with a positive probability $1-\phi$, we have $A^* \leq \frac{1}{2} n \etto$. Gaussian limits of the probability for $\cA(0)$ to almost percolate have been derived in the case of classical bootstrap percolation on $\gnp$ by Janson et al. in \cite{JLTV}. Their proof of Theorem 3.6 in \cite{JLTV} might be adapted to the setting of majority bootstrap percolation.

We showed also that the case $p=\frac{c}{n}$ has a specific behaviour where the activation spreads to a larger part of the graph but does not spread to almost all the graph. We could not determine the exact size of the final set of active vertices $| \cA^*| = A^*$.
A sharp estimate of the probability of receiving a mark at time $s$ denoted $p_s$ is necessary in this case too.
Moreover, a study of the function $f_{c,\theta}(x)$ which gave an approximation of $\frac{A(xn)-xn}{n}$ might show for different values of $c$ and $\theta$, various number of roots and the appearance of a double root for some critical values $\theta(c)$ for a given $c = pn$. Such a behaviour has already been noticed and treated on classical bootstrap percolation on the random graph $\gnp$ in \cite{JLTV}.

Finally, our proof of Theorem \ref{theo:p>>1overn} \ref{theo:p>>1overniii} shows that, 
under the condition of the theorem, the activation spreads to almost all the graph in only 1 generation. However, the total number of generations is not determined here.

\section*{Acknowledgements}
The authors thank Svante Janson for helpful comments.


\end{document}

\bibitem{AizenmanL}
\textsc{Aizenman, M.} and  \textsc{Lebowitz, J. L.} (1988).
Metastability effects in bootstrap percolation.
\textit{J. Phys. A} \textbf{21}, no. 19, 3801--3813.
\MR{0968311}

\bibitem[Amini(2010)]{Amini}
 \textsc{Amini, H.} (2010).
Bootstrap percolation and diffusion in random graphs with given vertex
degrees.
\textit{Electron. J. Combin.} \vol{17}, R25.
\MR{2595485}

\bibitem[von Bahr and Martin-L\"of(1980)]{vBahrML}
\textsc{von Bahr, B.} and \textsc{ Martin-L\"of, A.} (1980).
Threshold limit theorems for some epidemic processes.
\textit{Adv. Appl. Probab.} \vol{12} \no2, 319--349.
\MR{0569431}

\bibitem[Ball and Britton(2005)]{BallBritton}
\textsc{Ball, F.} and  \textsc{Britton, T.} (2005).
An epidemic model with exposure-dependent severities.
\textit{J. Appl. Probab.}  \vol{42},  no. 4, 932--949.
\MR{2203813}

\bibitem[Ball and Britton(2009)]{BallBritton09}
\textsc{Ball, F.} and  \textsc{Britton, T.} (2009).
An epidemic model with infector and exposure dependent severity.
\textit{Math. Biosci.}  \vol{218},  no. 2, 105--120.
\MR{2513676}

\bibitem[Balogh and Bollob{\'a}s(2006)]{BaloghBollobas}
\textsc{Balogh, J.}  and  \textsc{Bollob{\'a}s, B.} (2006).  
Bootstrap percolation on the hypercube.
\PTRF  \textbf{134},  no. 4, 624--648.
\MR{2214907}

\bibitem{BaloghBD-CM}
\textsc{Balogh, J.}, \textsc{Bollob{\'a}s, B.} \textsc{Duminil-Copin,
  H.} and \textsc{Morris, R.} (2010).
 The sharp threshold for bootstrap percolation
in all dimensions.
\textit{Trans. Amer. Math. Soc.}, to appear.
arXiv:1010.3326v2

\bibitem{BBM}
\textsc{Balogh, J.}, \textsc{Bollob{\'a}s, B.} and \textsc{Morris, R.}
(2009). Majority bootstrap percolation
on the hypercube.
\textit{Combin. Probab. Computing} \textbf{18}, 17--51.
\MR{2497373}

\bibitem{BaloghBM2}
\textsc{Balogh, J.}, \textsc{Bollob{\'a}s, B.} and \textsc{Morris, R.}
(2009).
Bootstrap percolation in three dimensions. \textit{Ann. Probab.} \textbf{37},
1329--1380.
\MR{2546747}

\bibitem{BaloghBM}
\textsc{Balogh, J.}, \textsc{Bollob{\'a}s, B.} and \textsc{Morris, R.}
(2010). Bootstrap
percolation in high dimensions. \textit{Combin. Probab. Computing} \textbf{19}, 643--692.
\MR{2726074}

\bibitem{BaloghBM3}
\textsc{Balogh, J.}, \textsc{Bollob{\'a}s, B.} and \textsc{Morris, R.}
(2011).
Graph bootstrap percolation.
Preprint.
arXiv:1107.1381

\bibitem{BaloghBMO}
\textsc{Balogh, J.}, \textsc{Bollob{\'a}s, B.},  \textsc{Morris, R.}
and  \textsc{Riordan, O.} (2011).
Linear algebra and bootstrap percolation.
Preprint.
arXiv:1107.1410

\bibitem{BaloghPeresPete}
\textsc{Balogh, J.},  \textsc{Peres, J.} and \textsc{Pete, G.} (2006).
Bootstrap percolation on infinite trees and non-amenable groups.
\textit{Combin. Probab. Comput.}  \textbf{15},  no. 5, 715--730.
\MR{2248323}

\bibitem{BaloghPete}
\textsc{Balogh, J.}  and \textsc{Pete, G.} (1998).
Random disease on the square grid.
\textit{Random Structures Algorithms} \textbf{13}, no. 3--4, 409--422.
\MR{1662792}

\bibitem{BaloghPittel}
\textsc{Balogh, J.} and \textsc{Pittel, B. G.} (2007).
Bootstrap percolation on the random regular graph.
\textit{Random Structures Algorithms} \textbf{30}, no. 1--2, 257--286.
\MR{2283230}

\bibitem{SJI}
\textsc{Barbour, A.D.},  \textsc{Holst, L.} and
\textsc{Janson, S.} (1992).
\textit{Poisson Approximation}.
Oxford Univ. Press, Oxford.

\bibitem{Bill}
\textsc{Billingsley, P} (1968).
\textit{Convergence of Probability Measures},
Wiley, New York.

\bibitem{Bo68}
\textsc{Bollob{\'a}s, B.} (1968).
Weakly $k$-saturated graphs. 
In \textit{Beitr\"age zur Graphentheorie (Kolloquium, Manebach, 1967)},
pp. 25--31,  Teubner, Leipzig, 1968.
\MR{0244077}

\bibitem[Bollob{\'a}s(2006)]{Boll:Art}
\textsc{Bollob{\'a}s, B.} (2006).
\textit{The Art of Mathematics. Coffee Time in Memphis.}
Cambridge University Press, New York.

\bibitem[Carroll(1876)]{snark}
\textsc{Carroll, L.} (1876).
\textit{The Hunting of the Snark}.
Macmillan, London.

\bibitem{CerfC}
 \textsc{Cerf R.} and   \textsc{Cirillo,  E.N.M.} (1999). Finite size scaling in three-dimensional
bootstrap percolation. \textit{Ann. Probab.}
\textbf{27}, 1837--1850.
\MR{1742890}

\bibitem{CerfManzo}
\textsc{Cerf R.} and  \textsc{Manzo, F.} (2002).
The threshold regime of finite volume bootstrap percolation.
\textit{Stochastic Process. Appl.} \textbf{101}, no. 1, 69--82.
\MR{1921442}

\bibitem{CerfManzo1}
 \textsc{Cerf R.} and  \textsc{Manzo, F.} (2010).
    A $d$-dimensional nucleation and growth model.
Preprint.
arXiv:1001.3990

\bibitem{CerfManzo2}
  \textsc{Cerf R.} and  \textsc{Manzo, F.} (2011).
Nucleation and growth for the Ising model in $d$  dimensions at very low
temperatures.
Preprint.
arXiv:1102.1741

\bibitem{ChalupaLR}
\textsc{Chalupa, J.}, \textsc{Leath, P. L.} and \textsc{Reich, G. R.} (1979).
Bootstrap percolation on a Bethe lattice.
\textit{J. Phys. C.} \textbf{12}, L31--L35.

\bibitem{D-CvE}
   \textsc{Duminil-Copin, H.} and \textsc{Van Enter, A. C. D. } (2010). Sharp metastability threshold
  for an anisotropic bootstrap percolation model.
Preprint.
 arXiv:1010.4691

\bibitem{FontesSS}
  \textsc{Fontes, L. R.}, \textsc{Schonmann, R. H.} and
  \textsc{Sidoravicius, V.} (2002).
  Stretched exponential fixation in stochastic Ising models at zero temperature.
\textit{Comm. Math. Phys.} \textbf{228}  \no3, 495--518.
\MR{1918786}

\bibitem{GravnerHM}
\textsc{Gravner, J.}, \textsc{Holroyd, A.E.} and \textsc{Morris, A.}  (2010).
A sharper threshold for
bootstrap percolation in two dimensions.
\PTRF\unskip, to appear.
arXiv:1002.3881v2

\bibitem{Gut}
\textsc{Gut, A.} (2005).
\textit{Probability: A Graduate Course}.
Springer, New York.

\bibitem{Holroyd}
\textsc{Holroyd, A.E.} (2003).
Sharp metastability threshold for two-dimensional bootstrap
percolation.
\PTRF \textbf{125}, no. 2, 195--224.
\MR{1961342}

\bibitem{HolroydLR}
\textsc{Holroyd, A.E.}, \textsc{Liggett,  T.M} and \textsc{Romik,
  D.} (2004).
Integrals, partitions, and cellular automata.
\textit{Trans. Amer. Math. Soc.} \textbf{356} (2004), 3349--3368.
\MR{2052953}

\bibitem[Janson(1994)]{SJII}
\textsc{Janson, S.}  (1994).
\textit{Orthogonal decompositions and functional limit theorems for
random graph statistics}.
Mem. Amer. Math. Soc., \vol{111}, no. 534,
Amer. Math. Soc., Providence, R.I.
\MR{1219708}

\bibitem{SJ215}
 \textsc{Janson, S.} (2009).
On percolation in random graphs with given vertex degrees.
\textit{Electronic J. Probab.} \textbf{14}, Paper 5, 86--118.
\MR{2471661}

\bibitem[Janson(2009)]{SJN6}
\textsc{Janson, S.} (2009).
Probability asymptotics: notes on notation.
\textit{Institute Mittag-Leffler Report} 12, 2009 spring.

\bibitem{JLR}
\textsc{Janson, S.}, \textsc{{\L}uczak, T.} and \textsc{Ruci{\' n}ski,
  A.} (2000).
\textit{Random Graphs}.
Wiley, New York.

\bibitem[Kallenberg(2002)]{Kallenberg}
\textsc{Kallenberg, O.} (2002).
\textit{Foundations of Modern Probability.}
2nd ed., Springer, New York.

\bibitem{Ko} \textsc{Kozma, R.},  \textsc{Puljic, M.}, \textsc{Balister, P.}
      \textsc{Bollob{\'a}s, B.} and
  \textsc{Freeman, W.} (2005).
  Phase transitions in the neuropercolation model of neural
  populations with mixed local and non-local interactions.
 \textit{Biol. Cybernet.} \textbf{92},  no. 6, 367--379.
\MR{2206497}

\bibitem[Martin-L\"of(1986)]{ML86}
\textsc{Martin-L\"of, A.} (1986).
Symmetric sampling procedures, general epidemic processes and their
threshold limit theorems.
\textit{J. Appl. Probab.}  \vol{23}  \no2,  265--282.
\MR{0839984}

\bibitem[Martin-L\"of(1998)]{ML98}
\textsc{Martin-L\"of, A.} (1998).
The final size of a nearly critical epidemic, and the first passage time of
a Wiener process to a parabolic barrier.
\textit{J. Appl. Probab.}  \vol{35} \no3 ,  671--682.
\MR{1659544}

\bibitem{Morris}
\textsc{Morris, R.} (2009).
Minimal percolating sets in bootstrap percolation.
\textit{Electron. J. Combin.}  \textbf{16},  no. 1, Research
Paper 2, 20 pp.
\MR{2475525}

\bibitem{Morris1}
\textsc{Morris, R.} (2011). Zero-temperature Glauber dynamics on $Z^d$.
\PTRF  \textbf{149}, 417--434.

\bibitem[Scalia-Tomba(1985)]{Scalia-Tomba}
\textsc{Scalia-Tomba, G.-P.} (1985).
Asymptotic final-size distribution for some chain-binomial processes.
\textit{Adv. Appl. Probab.}  \vol{17},  no. 3, 477--495.
\MR{0798872}

\bibitem{Schonmann}
\textsc{Schonmann, R.H.} (1992). On the behaviour of some cellular automata related to
bootstrap percolation. \textit{Ann. Probab.}  \vol{20}, 174--193.
\MR{1143417}

\bibitem[Sellke(1983)]{Sellke}
\textsc{Sellke, T.} (1983).
On the asymptotic distribution of the size of a stochastic epidemic.
\textit{J. Appl. Probab.}  \vol{20}  no. 2, 390--394.
\MR{0698541}

\bibitem{Tl}
  \textsc{Tlusty, T.} and \textsc{ Eckmann,  J.-P.} (2009). Remarks on bootstrap percolation
  in metric
  networks.  \textit{J. Phys. A: Math. Theor.} {\bf 42}, 205004.
\MR{2515591}

\bibitem{T}
\textsc{Turova, T.} (2011). The emergence of connectivity in neuronal networks: from
  bootstrap percolation to auto-associative memory. 
  \textit{Brain Research}, to appear.

\bibitem{TV}
\textsc{Turova, T.} and   \textsc{ Villa, A.} (2007).
On a phase diagram for random neural networks with embedded
                spike timing dependent plasticity.
\textit{BioSystems}
\textbf{89},  280--286.

\bibitem{Vallier:PhD}
\textsc{Vallier, T.} (2007).
\textit{Random graph models and their applications}.
Ph. D. thesis, Lund University.
\end{thebibliography}

\end{document}